\documentclass[12pt]{amsart}
\usepackage{amsmath,amsthm,amsfonts,amssymb}
\usepackage{mathrsfs}
\usepackage{tikz}

\usepackage{float}

\usepackage{algorithmic}
\let\chapter\section               
\usepackage{algorithm}  

\usepackage{color}
\usepackage[
	hyperindex,
	pagebackref,
	pdftex,
	pdftitle={Fusion rules of type A_2},
	pdfdisplaydoctitle,
	pdfpagemode=UseNone,
	breaklinks=true,
	extension=pdf,
	bookmarks=false,
	plainpages=false,
	colorlinks,
	linkcolor=linkblue,
	citecolor=linkblue,
	urlcolor=linkblue,
	pdfmenubar=true,
	pdftoolbar=true,
	pdfpagelabels,
	pdfpagelayout=SinglePage,
	pdfview=Fit,
	pdfstartview=Fit
]{hyperref}
\definecolor{linkblue}{rgb}{0,0.2,0.6}
\newcommand{\neturl}[1]{\href{#1}{{\sffamily{\texttt{#1}}}}}
\usepackage[backrefs,nobysame,alphabetic]{amsrefs}

\newtheorem{theorem}{Theorem}[section]
\newtheorem{proposition}[theorem]{Proposition}

\newtheorem{lemma}[theorem]{Lemma}
\theoremstyle{definition}

\newcommand{\CC}{\mathbb{C}}
\newcommand{\RR}{\mathbb{R}}

\newcommand{\rank}{\operatorname{rank}}
\newcommand{\sgn}{\operatorname{sgn}}
\newcommand{\sL}{\mathfrak{sl}}

\begin{document}

\pagenumbering{arabic}
\title{A new proof of a formula for the type $A_2$ fusion rules}
\date{\today}

\author{Amy Barker}

\author{David Swinarski}

\author{Lauren Vogelstein}

\author{John Wu}

\address{Department of Mathematics, Fordham University, New York, NY 10023, USA}
\email{dswinarski@fordham.edu}

\begin{abstract}We give a new proof of a formula for the fusion rules
  for type $A_2$ due to B{\'e}gin, Mathieu, and Walton.  Our approach is
  to symbolically evaluate the Kac-Walton algorithm.  
\end{abstract}

\maketitle

\section{Introduction}
For an affine Lie algebra $\widehat{\mathfrak{g}}$, the irreducible integrable $\widehat{\mathfrak{g}}$-modules are
classified by a highest weight and an integer $\ell$ called the
level.  The tensor product on $\widehat{\mathfrak{g}}$-modules is
additive with respect to the level.  There exists a second product 
called the fusion product, which is level-preserving.

The fusion rules of an affine Lie algebra are the full set of
structure constants $N_{\lambda,\mu}^{(\ell) \nu}$ that describe how the fusion product of two
irreducible integrable level $\ell$ $\widehat{\mathfrak{g}}$-modules decomposes into irreducibles.  
Kac and Walton independently found an algorithm for computing the
fusion rules.  The Kac-Walton algorithm only uses the combinatorics of
the underlying root system, and hence, this algorithm can be used to
define a product on $\mathfrak{g}$-modules as well
as $\widehat{\mathfrak{g}}$-modules.  In this case, the algorithm is
highly similar to the
Racah-Speiser algorithm for tensor product decompositions, which is an
algorithmic version of a formula that is variously attributed to 
Brauer, Klimyk, Steinberg, and Racah; see Section \ref{Kac Walton
  section} for more discussion.  

For Type $A_1$, the fusion rules for any level are easily
computed.  For Type $A_2$, B{\'e}gin, Mathieu, and Walton give a closed formula for
the fusion rules for any level in \cite{BMW}.  For other
root systems, the fusion rules are known in some special cases.  For
instance, when the root system rank and level are small, the fusion rules can be
computed using a computer; if the level is small, level-rank
duality may be used; and if the weights have special properties, 
additional formulas are known \cites{MorseSchilling2012,SchillingShimozono2001,Tudose2002}.  But at the time of this writing,
we do not know of any other root systems besides $A_1$ and $A_2$ where the fusion
rules are fully known for all weights and levels. 

B{\'e}gin, Mathieu, and Walton derive their formula for the
fusion rules of type $A_2$ using another formula called the depth
rule.  At the time their paper was published, the depth rule was only
a conjecture, but it has since been proven in \cite{FeingoldFredenhagen2008}.  Unfortunately,
extending the approach used in \cite{BMW} to other root systems has proven difficult.

In this paper, we give a new proof of B{\'e}gin, Mathieu, and Walton's formula for the
fusion rules of type $A_2$.  Our approach is to symbolically evaluate
the Kac-Walton algorithm using the computer algebra system
\texttt{Macaulay2}.  We hope that our approach can be
applied to obtain fusion rules for some other root systems of small
rank.  

We briefly mention three applications of fusion rules.

One application of the fusion rules is to compute the ranks of vector
bundles of conformal blocks \cite{Beauville}.  Write
$\mbox{}^{*}$ for the involution on the weight lattice given by
$-w_0$, where $w_0$ is the
longest word in the Weyl group.  The ranks of conformal blocks on
$\overline{M}_{0,3}$ are related to fusion coefficients by
\[ \rank \mathbb{V}(\mathfrak{g}, \ell,(\lambda,\mu,\nu)) = N_{\lambda,\mu}^{(\ell) \, \nu^{*}}.
\]
Then, for any $g$
and $n$ with $3g-3+n \geq 0$, factorization of vector bundles of conformal blocks allows the rank
of any conformal block $\mathbb{V}(\mathfrak{g},
\ell,\vec{\lambda})$ on $\overline{\mathcal{M}}_{g,n}$ to be computed
recursively with the fusion rules as the seeds of this recursion.

As a second application, the fusion rules are related to the quantum
cohomology of Grassmannians, at least in type A.  Specifically, the
ring $\mathcal{F}(\widehat{\mathfrak{sl}}(n))_{k}$ with
generators indexed by the irreducible integrable level $k$
$\widehat{\mathfrak{sl}}(n)$-modules 
and structure constants given by the fusion rules is a quotient of the
small quantum cohomology ring $qH^{\bullet}(\operatorname{Gr}_{k,n+k})$
(\cite{KorffStroppel}).

Finally, since the fusion coefficients $N_{\lambda,\mu}^{(\ell),\nu}$
are always dominated by the tensor coefficients
$N_{\lambda,\mu}^{\nu}$, we may view the fusion product as a truncated
tensor product.  It seems worth investigating whether fusion products
could be used to approximate tensors in scientific or engineering applications.  

\subsection{Outline of the paper}
In Section \ref{BMW section} we present a formula for the fusion rules
due to B{\'e}gin, Mathieu, and Walton.  In Section \ref{Kac Walton
  section} we review the Racah-Speiser and Kac-Walton algorithms.  In
Section \ref{our proof section} we give our proof of the B{\'e}gin-Mathieu-Walton formula.

\subsection{Acknowledgements}   
The first and third authors were supported by scholarships and summer
research funding from the Clare Boothe Luce Foundation.  The fourth author was
supported by summer research funding from the dean of Fordham College
at Lincoln Center.  The second author would like to thank John Cannon
and the \texttt{Magma} group for hosting a visit to the University of
Sydney during which the fusion rules were first implemented in
\texttt{Macaulay2}.  The second author would also like to thank Allen Knutsen
and Dan Roozemond for many helpful conversations, Mark Walton for
telling him about the reference \cite{FeingoldFredenhagen2008}, and Dan Grayson and
Mike Stillman for their advice in implementing the fusion rules in
\texttt{Macaulay2}.  Several additional \texttt{Macaulay2} packages were used in
our research, and we would like to thank their authors: Greg Smith, author of the
\texttt{FourierMotzkin} package \cite{FourierMotzkin}; Ren\'{e}
Birkner, author of  the \texttt{Polyhedra} package \cite{Polyhedra}; and Josephine Yu, Nathan
Ilten, and Qingchun Ren, who shared preliminary versions of their
\texttt{PolyhedralObjects} and \texttt{PolymakeInterface} packages.

\section{The B{\'e}gin-Mathieu-Walton formula} \label{BMW section}

\subsection{Notation}  Let $\mathfrak{g} = \sL_3$.  Let $\mathfrak{h} \subset \mathfrak{g}$
be the Cartan subalgebra of diagonal matrices.  Let
$\varepsilon_i: \mathfrak{h} \rightarrow \CC$ be the function
$\varepsilon_{i}(H) = h_{i,i}$.  Let $\alpha_1 =
\varepsilon_i-\varepsilon_2$ and $\alpha_2 =
\varepsilon_2-\varepsilon_3$.    Then $\Delta = \{\alpha_1,\alpha_2\}$
is a base of the root system of $\mathfrak{g}$, and 
$\theta = \alpha_1 + \alpha_2$ is the highest root with respect to $\Delta$.  The Cartan matrix is 
\[
\left[ \begin{array}{rr}
2 & -1 \\
-1 & 2
\end{array}\right].
\]
Let $\omega_1$ and $\omega_2$ be the fundamental dominant weights.
Then we have $\alpha_1 = 2\omega_1 - \omega_2$, $\alpha_2 = -\omega_1
+ 2 \omega_2$, and we may invert this system of equations to obtain $\omega_1 =
\frac{2}{3} \alpha_1 + \frac{1}{3} \alpha_2$ and $\omega_2 =
\frac{1}{3} \alpha_1 + \frac{2}{3} \alpha_2$.  The Killing form on the
fundamental weights is 
\begin{align*}
(\omega_1, \omega_1) &= \frac{2}{3}\\
(\omega_1, \omega_1) &= \frac{1}{3}\\
(\omega_2, \omega_2) &= \frac{2}{3}\\
\end{align*}
We have $(\theta,\theta) = 2$, and $(a \omega_1 + b \omega_2, \theta)
= a+b$.  

Let $\lambda = a \omega_1 + b \omega_2 = (a,b)$, $\mu = c \omega_1 + d
\omega_2=(c,d)$, and $\nu = e \omega_1 + f \omega_2 = (e,f)$.  Let
$C^{+}$ be the fundamental Weyl chamber; then $C^{+} = \{ c_1 \omega_1
+ c_2 \omega_2 : c_1, c_2 \geq 0\}$.  The
fundamental Weyl alcove of level $\ell$ is $P_{\ell} = \{ \beta \in
C^{+} : (\beta,\theta) \leq \ell\}$.  Thus, $\lambda, \mu, \nu \in
P_{\ell}$ if and only if $a,b,c,d,e,f \geq 0$ and $a+b,c+d,e+f \leq \ell$.

\subsection{The B\'{e}gin-Mathieu-Walton formula} In the exposition below, we combine some of the formulas from
\cite{BMW} to obtain a more self-contained presentation.
\begin{theorem}{\cite{BMW}}
The fusion rules of type $A_2$ are given as follows:
\begin{eqnarray*}
N_{\lambda,\mu}^{(\ell) \nu}  = &  
\left\{
\begin{array}{ll}
\min \{ k_0^{\max}, \ell \} - k_0^{\min} +1 & \mbox{if $\ell \geq
  k_0^{\min}$ and $N_{\lambda,\mu}^{\nu} >0$ }, \\
0 & \mbox{if $\ell < k_0^{\min}$ or $N_{\lambda,\mu}^{\nu} = 0$}, 
\end{array}
\right.
\end{eqnarray*}
where 
\begin{eqnarray*}
A &= &\frac{1}{3} ( 2(a+ c + f) + (b+d+e)), \\
B &= &\frac{1}{3} ( (a + c + f) + 2(b+d+e)), \\
k_0^{\min} & = & \max \{ a+b, c+d, e+f,
A - \min(a,c,f), B -
\min(b,d,e)\}, \\
k_0^{\max} & = & \min \{ A, B \}, \\
\delta & = &  \left\{
\begin{array}{ll}
1 & \mbox{ if $k_0^{\max} \geq k_0^{\min}$ and $A,
  B \in \mathbb{Z}_{\geq 0}$,} \\
0 & \mbox{ otherwise,}
\end{array}
\right.  \\
N_{\lambda,\mu}^{\nu} & = & (k_0^{\max} - k_0^{\min} +1) \delta.\\
\end{eqnarray*}
\end{theorem}

%

We note three nice features of this formula.  
First, note that $\ell$ enters only in the very last step of the
calculation; it does not appear in the definition of $A$,
$B$, $k_{0}^{\min}$, $k_{0}^{\max}$, $\delta$, or
$N_{\lambda,\mu}^{\nu}$.  Second, it is clear that
$N_{\lambda,\mu}^{(\ell) \nu}$ stabilizes for all sufficiently large
$\ell$, specifically once $\ell \geq k_{0}^{\max}$ and $\ell \geq k_{0}^{\min}$.  Finally, this
formula allows us to interpret $N_{\lambda,\mu}^{(\ell) \nu}$ as the
number of lattice points in a polytope; specifically, it is the number
of integers $x$ satisfying $k_{0}^{\min} \leq x \leq \min\{ k_{0}^{\max},\ell\}$.

\subsection{An equivalent version of the B\'{e}gin-Mathieu-Walton formula}
We modify the formula from B{\'e}gin, Mathieu, and Walton's paper
slightly.  We use fewer instances of $\max$ and $\min$, and the cases are
rewritten slightly to match the output we obtain from the Kac-Walton algorithm.  

Define
\begin{eqnarray*}
G(\lambda,\mu,\nu,\ell) := 
\left\{
\begin{array}{ll}
\ell_0^{\max} - k_0^{\min} +1 & \mbox{if $\ell_0^{\max} - k_0^{\min}
  \geq -1$},\\
0 & \mbox{otherwise}. 
\end{array}
\right.
\end{eqnarray*}
where 
\begin{eqnarray*}
A &= &\frac{1}{3} ( 2(a + c + f) + (b+d+e)), \\
B &= &\frac{1}{3} ( (a + c + f) + 2(b+d+e)), \\
k_0^{\min} & := & \max \{ a+b, c+d, e+f,
A -a, A -c, A -f,
B -b, B -d, B -e\}, \\
\ell_0^{\max} & = & \min \{ A, B, \ell \}.
\end{eqnarray*}

\begin{proposition}[B\'{e}gin-Mathieu-Walton] \label{main proposition}
If $A$ and $B$ are integers, then $G(\lambda,\mu,\nu,\ell)  = N_{\lambda,\mu}^{(\ell) \, \nu} $.
\end{proposition}

We view the formula $G(\lambda,\mu,\nu,\ell)$ above as a continuous piecewise linear
function supported on 27 polyhedral cones.  
As an example of one such cone, to get the expression 
$N_{\lambda,\mu}^{(\ell) \nu} = A-a-b+1$ above requires that 
\begin{gather*}
a+b = \max \{ a+b, c+d, e+f,
A -a, A -c, A -f,
B -b, B -d, B -e\}, \\
\ell = \min \{ A, B, \ell \}.
\end{gather*}
This leads to the 
the inequalities 
\begin{align*}
a+b & \geq c+d \\
a+b & \geq e+f \\
a+b & \geq A-a\\
a+b & \geq A-b\\
a+b & \geq A-f\\
a+b & \geq B-b\\
a+b & \geq B-d\\
a+b & \geq B-e\\
A & \leq B \\
A & \leq \ell.
\end{align*}
We also have $a,b,c,d,e,f,l \geq 0$ and $a+b, c+d, e+f \leq \ell$.
These 20 inequalities determine a polyhedral cone in $\RR^7$.

In a similar fashion, we may associate a finitely-generated polyhedral
cone to each of the remaining 26 nonzero
expressions that may arise from $G(\lambda,\mu,\nu,\ell).$

\section{The Racah-Speiser and Kac-Walton algorithms} \label{Kac Walton
  section}

Two references for the Kac-Walton algorithm are \cite{Kac}*{Exercise
  13.35} and \cite{Walton}.  Its history is described in
\cite{Walton}.  The Kac-Walton algorithm is closely related to the Racah-Speiser
algorithm for tensor coefficients, and so we recall the Racah-Speiser
algorithm first.  

\subsection{The Racah-Speiser algorithm} 
Let $\mathfrak{g}$ be a Lie algebra.  Let $\mathfrak{h} \subseteq
\mathfrak{g}$ be a Cartan subalgebra, let $\Phi$ be the root system
determined by $\mathfrak{h}$, and let $\Delta = \{ \alpha_1,
\ldots, \alpha_n\}$ be a base of $\Phi$.  Let $\lambda$ be a dominant
integral weight, and let $V(\lambda)$ be an irreducible
finite-dimensional $\mathfrak{g}$-module with highest
weight $\lambda$.  

Let $m_{\lambda}(\mu)$ denote the dimension of the weight space
$V_{\mu}$ in the irreducible representation $V(\lambda)$.  
We shall refer to the set of pairs
$(\mu,m_{\lambda}(\mu))$ as the \textit{weight diagram} of $\lambda$;
clearly, the character of $V(\lambda)$ can be computed from the weight
diagram, and vice versa.   Let $\Phi^{+}$ be the set of positive roots, and let
$\rho = \frac{1}{2} \sum_{\alpha \in \Phi^{+}} \alpha$.  Let $C^{+}$ be the fundamental Weyl chamber.  

Let $W$ be the Weyl group of $\mathfrak{g}$, and for $w \in W$, let 
$w \cdot \varphi$ be the shifted reflection defined by $w \cdot \varphi= w(\varphi+\rho)-\rho$.

Define
the tensor product coefficients $N_{\lambda \mu}^{\nu} $ by 
\[
V(\lambda) \otimes  V(\mu) = \bigoplus_{\nu \in C^{+}}  V(\nu)^{\oplus
  N_{\lambda \mu}^{\nu}}.  
\]

The Racah-Speiser algorithm is described below: 

\begin{algorithm}[H] \label{Racah-Speiser algorithm}
\caption{Racah-Speiser algorithm}
\begin{algorithmic}
\REQUIRE dominant integral weights $\lambda$ and $\mu$
\ENSURE the set of tensor product coefficients $\{ N_{\lambda
  \mu}^{\nu} \}$
\STATE Begin with $N_{\lambda \mu}^{\nu} = 0$.
\STATE Compute $WD(\lambda)$, the weight diagram of $\lambda$.  
\STATE Translate each weight in $WD(\lambda)$ by $\mu$.
\STATE For each weight $\varphi$ in $WD(\lambda) +\mu$, if $\varphi$
is not fixed by any shifted reflection $w \cdot \varphi$ for $w \in
W$, compute an element $w$ such that $w \cdot \varphi \in C^{+}$ and add $m_{\lambda}(\varphi-\mu) \sgn(w)$ to $N_{\lambda \mu}^{w \cdot \varphi}$.
\RETURN $\{ N_{\lambda \mu}^{\nu} \}$
\end{algorithmic}
\end{algorithm}

For a proof of the correctness of this algorithm, we refer to \cite{GoodmanWallach}*{Corollary 7.1.7}; see
also \cite{Humphreys}*{Exercise 24.9}, where this formula is
attributed to Brauer-Klimyk, and \cite{FultonHarris}*{Exercise 25.31},
where this formula is attributed to Racah.  After converting Goodman
and Wallach's notation to ours, the main formula of Corollary 7.1.7 is 
\[
N_{\lambda \mu}^{\nu} = \sum_{w \in W} \operatorname{sgn}(w) m_{\lambda}(\nu + \rho - w(\mu+\rho)).
\]
Since the weight diagram is symmetric under the Weyl group, we have 
\begin{align*}
m_{\lambda}(\nu+\rho -w
(\mu+\rho)) & = m_{\lambda} (w(\nu+\rho) -
(\mu+\rho)) \\
&= m_{\lambda} (w \cdot \nu - \mu).
\end{align*}
Substituting this into the previous formula gives 
\begin{equation} \label{GoodmanWallach}
N_{\lambda \mu}^{\nu} = \sum_{w \in W} \sgn(w) m_{\lambda} (w \cdot \nu - \mu),
\end{equation}
and this formula agrees with the calculations described in the Racah-Speiser algorithm.

\textit{Example.}  As an example, let $\mathfrak{g}= \mathfrak{sl}_3$, and let $\lambda =
(4,2) = 4 \omega_1  + 2 \omega_2$ and $\mu = (3,1) = 3 \omega_1 + 1
\omega_2$.  We use the Racah-Speiser algorithm to compute the
decomposition of the tensor
product $V(\lambda) \otimes V(\mu)$ into irreducible $\mathfrak{sl}_3$
modules.  
\begin{figure}[H]
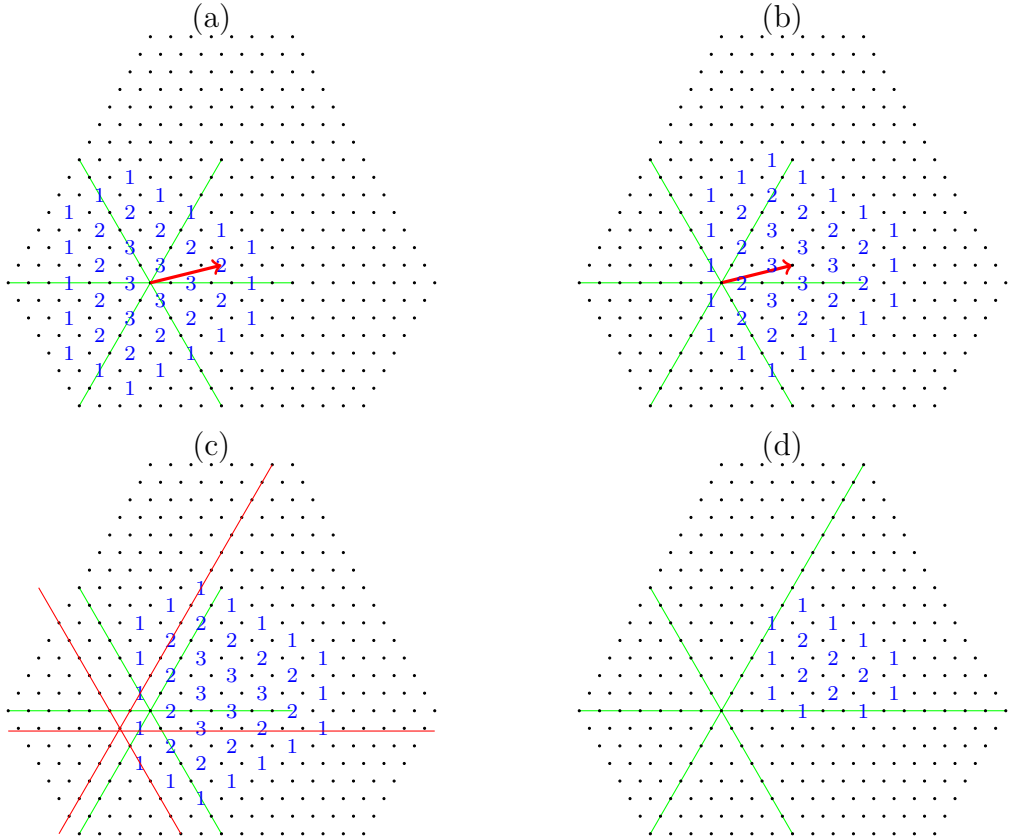
 
\caption{The Racah-Speiser algorithm.  (a) The weight diagram of
  $\lambda$.  (b) The weight diagram of $\lambda$ translated by
  $\mu$.  (c) Reflect into the fundamental chamber.  (d) The tensor
  coefficients.}
\label{Racah Speiser example}
\begin{center}
\include{RacahSpeiser-A2-4-2-3-1}
\end{center}
\end{figure}
From this we see that $N_{\lambda \mu}^{\nu} = 1$ if $\nu \in \{(0, 5), (1, 3), (1, 6), (2, 1), (3, 5), (4, 0), (5, 4), (7, 0), (7, 3), (8, 1) \}$,
$N_{\lambda \mu}^{\nu} = 2$ if $\nu \in \{ (2, 4), (3, 2), (4, 3), (5,
1), (6, 2)\}$, and otherwise $N_{\lambda \mu}^{\nu} = 0$.

\subsection{The Kac-Walton algorithm}  
Fix an integer $\ell \geq 0$.  The Kac-Walton algorithm differs from the Racah-Speiser algorithm by
replacing the Weyl group $W$ with the affine Weyl group
$\widehat{W}$ and the fundamental Weyl chamber $C^{+}$ by the
fundamental Weyl alcove $P_{\ell}$ defined below.  In contrast with the Weyl group,
the affine Weyl group is infinite.  However, it can be obtained by adding just one extra
generator to the Weyl group.  For $i=1,\ldots,n$, let $s_{i}$ be  the reflection across the hyperplanes perpendicular to the simple
root $\alpha_i$.  Let $s_{0}$ be the affine linear transformation
\[
s_{0}(\beta) = \beta + (\ell-(\beta,\theta)+1)\theta,
\]
where $\theta$ is the highest root, and $(-,-)$ is the Killing form
normalized so that $(\theta,\theta)=2$.  Then $W = \langle s_1,
\ldots, s_n \rangle$, and $\widehat{W} = \langle s_0,
\ldots, s_n \rangle$.

 The
fundamental Weyl alcove of level $\ell$ is $P_{\ell} = \{ \beta \in
C^{+} : (\beta,\theta) \leq \ell\}$.

The Kac-Walton algorithm is described below: 

\begin{algorithm}[H] 
\caption{Kac-Walton algorithm}
\label{Kac-Walton algorithm}
\begin{algorithmic}
\REQUIRE dominant integral weights $\lambda, \mu \in P_{\ell}$ 
\ENSURE the set of fusion coefficients $\{ N_{\lambda
  \mu}^{(\ell) \nu} \}$
\STATE Begin with $N_{\lambda \mu}^{(\ell) \nu} = 0$.
\STATE Compute $WD(\lambda)$, the weight diagram of $\lambda$.  
\STATE Translate each weight in $WD(\lambda)$ by $\mu$.
\STATE For each weight $\varphi$ in $WD(\lambda) +\mu$, if $\varphi$
is not fixed by any shifted reflection $w \cdot \varphi$ for $w \in
\widehat{W}$, compute an element $w$ such that $w \cdot \varphi \in
P_{\ell}$ and add $m_{\lambda}(\varphi-\mu) \sgn(w)$ to $N_{\lambda
  \mu}^{(\ell) w \cdot \varphi}$.
\RETURN $\{ N_{\lambda \mu}^{(\ell) \nu} \}$
\end{algorithmic}
\end{algorithm}

For the purposes of this paper, we shall use the Kac-Walton algorithm
to define the fusion coefficients $N_{\lambda \mu}^{(\ell) \nu}$.  For
a proof that the Kac-Walton algorithm computes the multiplicities of
irreducible level $\ell$ integrable $\widehat{g}$-modules in the
fusion product, see \cite{Kac}*{Exercise
  13.35} and \cite{Walton}.

\textit{Example.}  As an example, let $\mathfrak{g}= \mathfrak{sl}_3$, and let $\lambda =
(4,2) = 4 \omega_1  + 2 \omega_2$ and $\mu = (3,1) = 3 \omega_1 + 1
\omega_2$, and let $\ell=7$.  We use the Kac-Walton algorithm to compute the
decomposition of the fusion
product $V(\lambda) \otimes_{7} V(\mu)$ into irreducible $\mathfrak{sl}_3$
modules.  The first two steps are the same as those of the
Racah-Speiser algorithm; see Figure \ref{Racah Speiser example} (a) and (b).  
\begin{figure}[H]
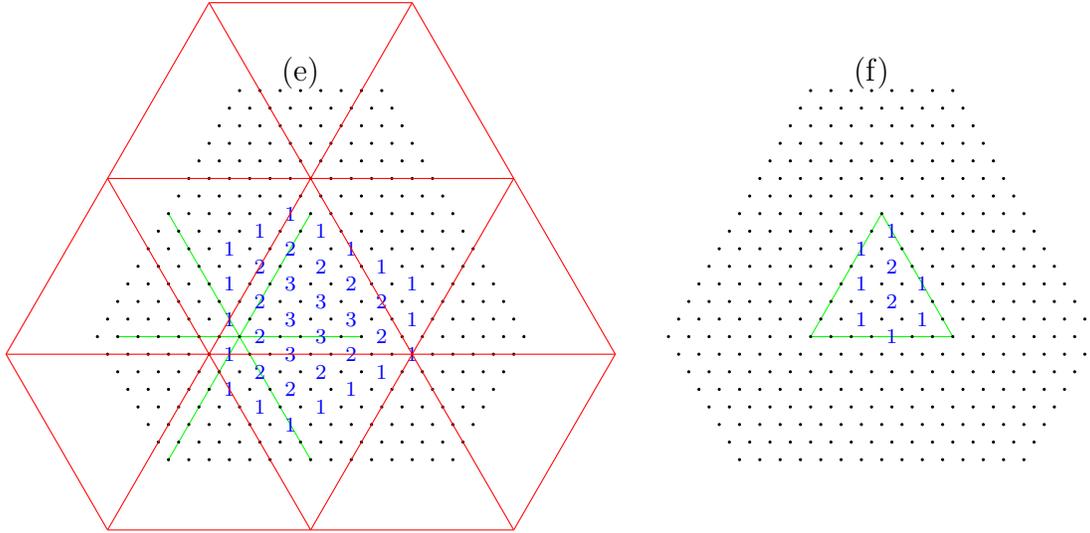

\caption{The Kac-Walton algorithm.  Steps (a) and (b) are the same as
  in Figure \ref{Racah Speiser example}.   (e) Reflect into the fundamental alcove.  (f) Fusion coefficients.}
\begin{center}
\include{KacWalton-A2-4-2-3-1-level-7}
\end{center}
\end{figure}
From this we see that $N_{\lambda \mu}^{(\ell) \nu} = 1$ if $\nu \in
\{(0, 5), (1, 3), (1, 6), (2, 1), (4, 0), (4, 3),  (5, 1) \}$,
$N_{\lambda \mu}^{(\ell) \nu} = 2$ if $\nu \in \{ (2, 4), (3, 2)\}$, and otherwise $N_{\lambda \mu}^{(\ell) \nu} = 0$.

\section{Our proof} \label{our proof section} 
\subsection{A multiplicity formula}
Weight diagrams for Type $A_2$ have a very pretty description; in 
\cite{Humphreys}*{\S 21.3}, Humphreys attributes this description
to Antoine and Speiser.  The boundary of the weight diagram is a
(nonregular) hexagon with all multiplicities equal to one.  As one passes
from one hexagonal ``shell'' of the weight diagram to the next ``shell'' inside it, the
multiplicity increases by one, until the shells become
triangles, at which point the multiplicity is constant.  See Figure
\ref{Racah Speiser example}(a) for an example.  

Writing formulas for the pattern described above yields the following:
\begin{lemma}[Antoine-Speiser] \label{Antoine-Speiser Lemma}
Let $\lambda=(a,b)$ and $\varphi = (x,y)$ be two weights in the
fundamental chamber $C^{+}$, and suppose $a+2b-x-2y$ is divisible by 3 (so that
$\lambda-\varphi$ is in the root lattice).  Then the multiplicity of $\varphi$ in $V(\lambda)$ is 
\begin{displaymath}
m_{\lambda}(\varphi) = \max \left\{ 0, \min\left \{ \frac{1}{3}(a+b-x-2y)+1,
  \frac{1}{3}(2a+b-2x-y)+1,a+1,b+1 \right\} \right\}.
\end{displaymath}
\end{lemma}
\begin{proof}
The most popular way to derive this formula is 
use the fact that the multiplicity of $\mu$ in $V(\lambda)$ is the
number of semistandard Young tableaux of shape $\lambda$ and weight
$\mu$ (see e.g. \cite{GoodmanWallach}*{Cor. 8.1.7}).  This leads to the inequalities printed above.  

However, following Exercise 25.15 in \cite{FultonHarris}, we wrote our
own proof using double induction and Freudenthal's formula (see
e.g.~\cite{FultonHarris}*{Lecture 25}).  The first induction is
on the distance to the boundary along a positive root, and the second
induction is on the distance from an arbitrary weight $\mu$ with distance $k$ to the boundary to the point $\lambda - k \theta$.  For the 
full proof, see our website:
\begin{center}
\neturl{http://faculty.fordham.edu/dswinarski/symbolickacwalton/}
\end{center}
\end{proof}

We view the formula in Lemma \ref{Antoine-Speiser Lemma} as a continuous piecewise linear function supported
on seven cones.  As an example of one such cone, to get the multiplicity expression
$\frac{1}{3}(a+b-x-2y)+1$ above requires the inequalities 
\begin{gather*}
\frac{1}{3}(a+b-x-2y)+1 \geq 0 \\
\frac{1}{3}(a+b-x-2y)+1 \leq \frac{1}{3}(2a+b-2x-y)+1\\
\frac{1}{3}(a+b-x-2y)+1 \leq a+1 \\
\frac{1}{3}(a+b-x-2y)+1 \leq b+1. 
\end{gather*}
These inequalities, together with the inequalities $0 \leq a,b,x,y$,
determine a finitely-generated polyhedral cone in $\RR^4$.
In a similar fashion, we associate three more cones to the other three nonzero
expressions in Lemma \ref{Antoine-Speiser Lemma}.

We define three different cones where the multiplicity expression is
0.  Observe first that since $\lambda \in P_{\ell}$, we have $a \geq 0$ and $b \geq 0$, so the expressions
$a+1$ and $b+1$ in Lemma \ref{Antoine-Speiser Lemma} never cause the
multiplicity to vanish.  Thus, we define one cone where $\frac{1}{3}(a+b-x-2y)+1 \geq 0$ and $\frac{1}{3}(2a+b-2x-y)+1
\leq 0$; in the second cone, we have $\frac{1}{3}(a+b-x-2y)+1 \leq 0$ and $\frac{1}{3}(2a+b-2x-y)+1
\geq 0$; and in the third cone, we have $\frac{1}{3}(a+b-x-2y)+1 \leq 0$ and $\frac{1}{3}(2a+b-2x-y)+1
\leq 0$.  Thus we obtain seven cones total covering the fundamental Weyl chamber.  

Since a weight diagram is symmetric under the Weyl group, we may use
the Weyl group to obtain expressions for the multiplicity in the remaining chambers.  This yields a formula with 42 cones.  However, the resulting 42
expressions are not distinct; some of these
cones may be combined, yielding the following formula, which has 14 cones.
\begin{proposition} \label{multiplicity cone expression}
If $\lambda-(x,y)$ is in the root lattice, then 
the multiplicity of $(x,y)$ in $V(\lambda)$ is given by the
continuous piecewise polynomial formula printed in Figure \ref{14 cone
  multiplicity formula}.
\begin{figure}[H] \label{14 cone multiplicity
    formula} \caption{Multiplicity expressions on 14 cones.  An expression of the form 
$F(x,y,a,b)$ in the left column represents the inequality $F(x,y,a,b) \geq 0$.}
\begin{center} \small
\begin{tabular}{|p{3.75in}|p{3in}|} \hline
Cone inequalities & Multiplicity \\ \hline
$x-y-a+b$, $x+2y-a+b$, $2x+y-2a-b-3$, & $0$\\ \hline
$-x+y+a-b$, $2x+y+a-b$, $x+2y-a-2b-3$& $0$\\ \hline
$-2x-y-a+b$, $x+2y-a+b$, $-x+y-2a-b-3$& $0$\\ \hline
$-x+y+a-b$, $-x-2y+a-b$, $-2x-y-a-2b-3$& $0$\\ \hline
$x-y-a+b$, $-2x-y-a+b$, $-x-2y-2a-b-3$& $0$\\ \hline
$2x+y+a-b$, $-x-2y+a-b$, $x-y-a-2b-3$& $0$\\ \hline
$x-y-a+b$, $x+2y-a+b$, $-2x-y+2a+b+3$, $2x+y+a-b$& $-(2/3)x-(1/3)y+(2/3)a+(1/3)b+1$\\ \hline
$-x+y+a-b$, $2x+y+a-b$, $x+2y-a+b$, $-x-2y+a+2b+3$& $-(1/3)x-(2/3)y+(1/3)a+(2/3)b+1$\\ \hline
$-2x-y-a+b$, $x+2y-a+b$, $x-y+2a+b+3$, $-x+y+a-b$& $(1/3)x-(1/3)y+(2/3)a+(1/3)b+1$\\ \hline
$-x+y+a-b$, $-2x-y-a+b$, $-x-2y+a-b$, $2x+y+a+2b+3$,& $(2/3)x+(1/3)y+(1/3)a+(2/3)b+1$\\ \hline
$x-y-a+b$, $-2x-y-a+b$, $x+2y+2a+b+3$, $-x-2y+a-b$& $(1/3)x+(2/3)y+(2/3)a+(1/3)b+1$\\ \hline
$x-y-a+b$, $2x+y+a-b$, $-x-2y+a-b$, $-x+y+a+2b+3$& $-(1/3)x+(1/3)y+(1/3)a+(2/3)b+1$\\ \hline
$-x+y+a-b$, $2x+y+a-b$, $-x-2y+a-b$, $a-b$ & $b+1$ \\ \hline
$-a+b+x-y$, $-a+b+x+2y$, $-a+b-2x-y$, $-a+b$ &  $a+1$ \\ \hline
\end{tabular}
\end{center} \normalsize
\end{figure}
\end{proposition}

\subsection{Contributing alcoves}
Recall that for $w \in \widehat{W}$, $w \cdot \beta =
w(\beta+\rho)-\rho$.  

\begin{lemma} \label{contributing alcoves lemma}
The alcove $w \cdot P_{\ell}$ contributes zero to the Kac-Walton
algorithm unless $w$ is equivalent in the Weyl group to one of the following 13 elements: 
\[
\{ s_0s_2s_0, s_0s_1s_0, s_1s_2s_1, s_0s_2, s_0s_1, s_2s_0, s_1s_0, s_2s_1, s_1s_2, s_0, s_2, s_1, Id
\}.
\]
\end{lemma}
\begin{proof}
Let $WD(\lambda)$ denote the weight diagram of $\lambda$, and let
$WP_{\ell}$ denote the $W$-orbit of $P_{\ell}$.  We hope that this
clash of notation will not cause too much confusion.  

Since $\lambda \in P_{\ell}$ and the weight diagram is symmetric under
the Weyl group, we have $WD(\lambda) \subset WP_{\ell}$.  Since $\mu \in P_{\ell}$, we
have $WD(\lambda) +\mu$ is contained in the Minkowski sum $\subset WP_{\ell} + P_{\ell}$, and we check in turn that
the Minkowski sum is contained in the union of the 13 alcoves listed.
In Figure \ref{Minkowski sum} below, the regions $P_{\ell}$,
$WP_{\ell}$, and $\subset WP_{\ell} + P_{\ell}$ are shown in
increasingly lighter shades of green, respectively, and the 13 alcoves
are labeled.
\end{proof}

\begin{figure}[H]
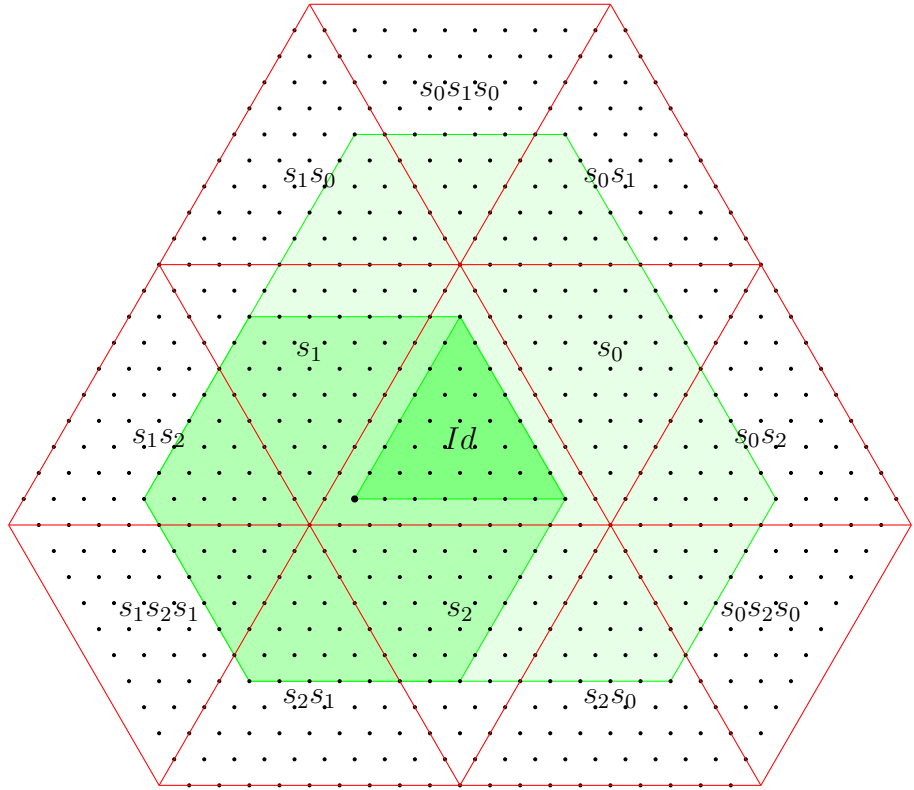
 
\caption{Alcoves contributing to the Kac-Walton algorithm}
\label{Minkowski sum}
\begin{center}
\include{MinkowskiSum}
\end{center}
\end{figure}
\subsection{Our \texttt{Macaulay2} types and functions}

We implemented two new types in \texttt{Macaulay2} called 
\texttt{ConeSupportedExpression} and
\texttt{ConeSupportedExpressionSet}.  These two types are highly specialized
for the calculations required here.  An object of type 
\texttt{ConeSupportedExpression} is a hash table recording an
expression and a cone on which it is supported.  An object of type
\texttt{ConeSupportedExpressionSet} is an unordered set of
\texttt{ConeSupportedExpression}s.  We assume that 
\begin{enumerate}
\item The dimension of each cone in each
  \texttt{ConeSupportedExpression} is equal to the dimension of the
  ambient vector space;
\item no two cones in a \texttt{ConeSupportedExpressionSet} have a
  full-dimensional intersection;
\item the union of the cones in a \texttt{ConeSupportedExpressionSet}
  is equal to the ambient vector space.  
\end{enumerate}

The multiplicity formula in Figure \ref{14 cone multiplicity
    formula} has these three properties, and hence can be implemented
  as an object of type \texttt{ConeSupportedExpressionSet}.

We implemented methods for adding two
\texttt{ConeSupportedExpressionSet}s and for multiplying a
\texttt{ConeSupportedExpressionSet} by a scalar.  

We also wrote a function \texttt{isUnionConvex} to decide whether the union
of several cones is convex.  One use of this function is to simplify a
\texttt{ConeSupportedExpressionSet}; if one nonzero expression is
supported on two or more cones, and the union of these cones is
convex, then we replace these cones by their union, yielding a
\texttt{ConeSupportedExpressionSet} containing fewer \texttt{ConeSupportedExpression}s.

\subsection{The main program}

We use the notation for roots and weights described in Section
\ref{BMW section}.

We begin with $\lambda = (a,b)$, $\mu = (c,d)$, and $\nu=(e,f)$.

For each word $w$ in the list of contributing alcoves in Lemma \ref{contributing alcoves lemma}, we compute
$w \cdot \nu -\mu= (x,y)$ and use the formulas in Figure \ref{14 cone multiplicity
    formula} to compute $m_{\lambda}(w \cdot \nu - \mu)$ as a
  \texttt{ConeSupportedExpressionSet}.  We then compute 
\[
N_{\lambda \mu}^{(\ell) \nu} = \sum_{w} \sgn(w) m_{\lambda}(w \cdot \nu
- \mu),
\]
simplifying the intermediate \texttt{ConeSupportedExpressionSet} after each addition or subtraction.  

The program takes approximately ten minutes to compute its answer.  It
finds 27 nonzero expressions supported on cones, and computes an
additional 82 cones supporting the expression 0.  

We checked that the 27 nonzero expressions we obtained and the cones
on which they are supported match the nonzero expressions and cones of B\'{e}gin, Mathieu, and Walton's
formula.  Since our program
computes its answer without using B\'{e}gin, Mathieu, and Walton's
formula along the way, we obtain a new, independent proof of Proposition \ref{main proposition}, first
proved by B\'{e}gin, Mathieu, and Walton in \cite{BMW}.  Notably, our proof does not
use the depth rule, which was used in \cite{BMW}.



\section*{References}
\begin{biblist}
\bib{Beauville}{article}{
   author={Beauville, Arnaud},
   title={Conformal blocks, fusion rules and the Verlinde formula},
   conference={
      title={},
      address={Ramat Gan},
      date={1993},
   },
   book={
      series={Israel Math. Conf. Proc.},
      volume={9},
      publisher={Bar-Ilan Univ.},
      place={Ramat Gan},
   },
   date={1996},
   pages={75--96},
   review={\MR{1360497 (97f:17025)}},
}

\bib{BMW}{article}{
   author={B{\'e}gin, L.},
   author={Mathieu, P.},
   author={Walton, M. A.},
   title={$\widehat{\rm su}(3)_k$ fusion coefficients},
   journal={Modern Phys. Lett. A},
   volume={7},
   date={1992},
   number={35},
   pages={3255--3265},
   issn={0217-7323},
   review={\MR{1191281 (93j:81028)}},
   doi={10.1142/S0217732392002640},
}

\bib{Polyhedra}{article}{
   author={Birkner, Ren\'{e}},
   title={\texttt{\upshape Polyhedra}: a package in
     \texttt{\upshape Macaulay2} for computations with convex polyhedra, cones, and fans},
   date={2010},
   note={Available on the Macaulay2 website},
}

\bib{FeingoldFredenhagen2008}{article}{
   author={Feingold, Alex J.},
   author={Fredenhagen, Stefan},
   title={A new perspective on the Frenkel-Zhu fusion rule theorem},
   journal={J. Algebra},
   volume={320},
   date={2008},
   number={5},
   pages={2079--2100},
   issn={0021-8693},
   review={\MR{2437644 (2009f:17042)}},
   doi={10.1016/j.jalgebra.2008.05.026},
}

\bib{FultonHarris}{book}{
   author={Fulton, William},
   author={Harris, Joe},
   title={Representation theory},
   series={Graduate Texts in Mathematics},
   volume={129},
   note={A first course;
   Readings in Mathematics},
   publisher={Springer-Verlag},
   place={New York},
   date={1991},
   pages={xvi+551},
   isbn={0-387-97527-6},
   isbn={0-387-97495-4},
   review={\MR{1153249 (93a:20069)}},
}

\bib{polymake}{article}{
   author={Gawrilow, Ewgenij},
   author={Joswig, Michael},
   title={\texttt{\upshape polymake}: a framework for analyzing convex polytopes},
   date={2012},
   note={Version 2.12, \neturl{http://www.math.tu-berlin.de/polymake/}},
}

\bib{GoodmanWallach}{book}{
   author={Goodman, Roe},
   author={Wallach, Nolan R.},
   title={Symmetry, representations, and invariants},
   series={Graduate Texts in Mathematics},
   volume={255},
   publisher={Springer, Dordrecht},
   date={2009},
   pages={xx+716},
   isbn={978-0-387-79851-6},
   review={\MR{2522486 (2011a:20119)}},
   doi={10.1007/978-0-387-79852-3},
}

\bib{Macaulay2}{article}{
   author={Grayson, Dan},
   author={Stillman, Mike},
   title={\texttt{\upshape Macaulay2}: a software system for research
     in algebraic geometry},
   date={2014},
   note={Version 1.6, \neturl{http://www.math.uiuc.edu/Macaulay2/}},
}

\bib{Humphreys}{book}{
   author={Humphreys, James E.},
   title={Introduction to Lie algebras and representation theory},
   series={Graduate Texts in Mathematics},
   volume={9},
   note={Second printing, revised},
   publisher={Springer-Verlag, New York-Berlin},
   date={1978},
   pages={xii+171},
   isbn={0-387-90053-5},
   review={\MR{499562 (81b:17007)}},
}

\bib{Kac}{book}{
   author={Kac, Victor G.},
   title={Infinite-dimensional Lie algebras},
   edition={3},
   publisher={Cambridge University Press, Cambridge},
   date={1990},
   pages={xxii+400},
   isbn={0-521-37215-1},
   isbn={0-521-46693-8},
   review={\MR{1104219 (92k:17038)}},
   doi={10.1017/CBO9780511626234},
}

\bib{KorffStroppel}{article}{
   author={Korff, Christian},
   author={Stroppel, Catharina},
   title={The $\widehat{\germ{sl}}(n)_k$-WZNW fusion ring: a
   combinatorial construction and a realisation as quotient of quantum
   cohomology},
   journal={Adv. Math.},
   volume={225},
   date={2010},
   number={1},
   pages={200--268},
   issn={0001-8708},
   review={\MR{2669352 (2012a:17022)}},
   doi={10.1016/j.aim.2010.02.021},
}

\bib{MorseSchilling2012}{article}{
   author={Morse, Jennifer},
   author={Schilling, Anne},
   title={A combinatorial formula for fusion coefficient},
   date={2012},
   eprint={http://arxiv.org/abs/1207.0786},
}



\bib{SchillingShimozono2001}{article}{
   author={Schilling, Anne},
   author={Shimozono, Mark},
   title={Fermionic formulas for level-restricted generalized Kostka
   polynomials and coset branching functions},
   journal={Comm. Math. Phys.},
   volume={220},
   date={2001},
   number={1},
   pages={105--164},
   issn={0010-3616},
   review={\MR{1882402 (2003k:05140)}},
   doi={10.1007/s002200100443},
}

\bib{FourierMotzkin}{article}{
   author={Smith, Greg},
   title={\texttt{\upshape FourierMotzkin}: a package in
     \texttt{\upshape Macaulay2} for convex hull and vertex enumeration},
   date={2008},
   note={Available on the Macaulay2 website},
}

\bib{LieTypes}{article}{
   author={Swinarski, David},
   title={\texttt{\upshape LieTypes}: a package in \texttt{\upshape Macaulay2} for
     calculations related to Lie algebras},
   date={2014},
   note={\neturl{http://faculty.fordham.edu/dswinarski/}},
}

\bib{Tudose2002}{book}{
   author={Tudose, Geanina},
   title={On the combinatorics of sl(n)-fusion algebra},
   note={Thesis (Ph.D.)--York University (Canada)},
   publisher={ProQuest LLC, Ann Arbor, MI},
   date={2002},
   pages={99},
   isbn={978-0612-72015-2},
   review={\MR{2703805}},
}

\bib{Walton}{article}{
   author={Walton, Mark A.},
   title={Algorithm for WZW fusion rules: a proof},
   journal={Phys. Lett. B},
   volume={241},
   date={1990},
   number={3},
   pages={365--368},
   issn={0370-2693},
   review={\MR{1055061 (91k:81180a)}},
   doi={10.1016/0370-2693(90)91657-W},
}

\end{biblist}
\end{document}